\documentclass[oneside,english]{amsart}
\usepackage{mathptmx}       
\usepackage{helvet}         
\usepackage{courier}        
\usepackage{type1cm}        
\usepackage{graphicx}        
\usepackage{hyperref}

\usepackage{enumitem}

\numberwithin{equation}{section}
\numberwithin{figure}{section}
\theoremstyle{plain}
\newtheorem{thm}{\protect\theoremname}
\theoremstyle{plain}
\newtheorem{prop}[thm]{\protect\propositionname}

\providecommand{\propositionname}{Proposition}
\providecommand{\theoremname}{Theorem}

\newcommand{\dd}{\textrm{d}}
\newcommand{\RR}{\mathbb{R}}
\newcommand{\bx}{\mathbf{x}}
\newcommand{\by}{\mathbf{y}}
\newcommand{\bt}{\mathbf{t}}
\newcommand{\Bm}{B_{3}^{\rm{meas}}}
\newcommand{\Bext}{B_{3}^{\rm{ext}}}

\usepackage{amsfonts}
\usepackage{amsmath}
\usepackage{amssymb}

\begin{document}

\title{A method to extrapolate the data for the inverse magnetisation problem with a planar sample}
\author{Dmitry Ponomarev$^1$} 
\thanks{$^1$ Centre Inria d'Universit{\'e} C{\^o}te d'Azur, Factas team, 2004 Route des Lucioles, 06902 Biot, France, \href{mailto:dmitry.ponomarev@inria.fr}{dmitry.ponomarev@inria.fr}}

\maketitle

\begin{abstract}
A particular instance of the inverse magnetisation problem is considered. It is assumed that the support of a magnetic sample (a source term in the Poisson equation in $\mathbb{R}^3$) is contained in a bounded planar set parallel to the measurement plane. Moreover, only one component of the magnetic field is assumed to be known (measured) over the same planar region in the measurement plane. We propose a method to extrapolate the measurement data to the whole plane relying on the knowledge of the forward operator and the geometry of the problem. The method is based on the spectral decomposition of an auxiliary matrix-function operator. The results are illustrated numerically.
\end{abstract} 

\section{Introduction}
Inverse magnetisation problems are inverse source problems for Poisson partial differential equation with the source term being a divergence of some vector field.

We are concerned with a particular instance of such problems which corresponds to a concrete experimental set-up that is used in the Paleomagnetism lab of Earth, Atmospheric \& Planetary Sciences Department at MIT (USA) and that can be schematically described as follows. 
Given measurements of vertical component of the magnetic field $B_3$ on a patch $Q\subset \mathbb{R}^2$ of the horizontal plane at height $x_3=h$, one wants to somehow estimate a planar magnetisation distribution $\overrightarrow{M}$ supported inside the region $Q$ of the underlying horizontal plane (i.e., the plane at height $x_3=0$). Such a problem is known to be ill-posed \cite{mag_kern, mag_kern_mom}
for it admits non-unique solutions, and the solutions are not stable with respect to the measured data $B_3$. To deal with the non-uniqueness issue, additional assumptions (such as unidimensionality) on the magnetisation distribution must be made, see \cite{mag_kern, mag_unidir}.
Alternatively, to avoid such a strong restriction, one can resort to reconstruction of a net magnetisation of a sample (also called magnetisation moment) $\overrightarrow{m}:=\iint_{Q}\overrightarrow{M}\dd^{2}\mathbf{x}$, a constant vector that can be recovered uniquely and is of primary physical importance. There have been several approaches proposed to estimate this quantity \cite{mom_BEP, mom_multip, mom_asympt}, however, the quality of estimation heavily depends on the size of the region with available data and the presence of noise. This leads us to the problem of data extrapolation and denoising.
In the present work, we propose one approach of dealing with this issue which is based on the spectral decomposition of some auxiliary operator that encodes the geometry of the problem and the assumption on support of the magnetisation distribution.     

Let us now formulate the problem mathematically. Given an open bounded set $Q\subset\mathbb{R}^2$, we assume that $\overrightarrow{M}$ $\equiv$ $\left(M_1,M_2,M_3\right)^{T}$ $\in$ $\left[W^{1,2}\left(\mathbb{R}^2\right)\right]^{2}\times L^{2}\left(\mathbb{R}^2\right)$, $\textrm{supp }\overrightarrow{M}\subset Q$ and $B_3^{\textrm{meas}}\in L^2\left(Q\right)$. In the case of non-noisy data, we have $\Bm=\left.B_3\right|_{Q\times\left\{h\right\}}$. The relation between the measured component of the magnetic field $B_3$ at height $x_3=h>0$ and the unknown magnetisation distribution $\overrightarrow{M}$ in the plane $x_3=0$ is given by (see, e.g., \cite[Sec 3.2]{mag_kern_mom} with $\mu_0=1$ by choice of units)
\begin{align}
B_{3}\left(\bx,h\right)= & -\frac{1}{2}\iint_{Q}\partial_{x_{1}}p_{h}\left(\bx-\bt\right)M_{1}\left(\bt\right)\dd^{2}\bt-\frac{1}{2}\iint_{Q}\partial_{x_{2}}p_{h}\left(\bx-\bt\right)M_{2}\left(\bt\right)\dd^{2}\bt\label{eq:B3_via_M1M2M3}\\
 & -\frac{1}{2}\iint_{Q}\partial_{h}p_{h}\left(\bx-\bt\right)M_{3}\left(\bt\right)\dd^{2}\bt,\qquad\bx\in\RR^2,\nonumber
\end{align}
where $p_h\left(\bx\right):=\frac{h}{2\pi\left(\left|\bx\right|^2+h^2\right)^{3/2}}$ and we adopted the notation $\mathbf{x}\equiv\left(x_1,x_2\right)^{T}$ for two-dimensional vectors.
\section{Field extrapolation}
We aim to extrapolate the field $\Bm$ from $Q$ to the whole $\RR^2$. First of all, let us show that this problem makes sense, at least when the measurement data are ideal (i.e., without noise).
\begin{prop}\label{prop:ext_uniq}
Given ideal measurement data $\Bm=\left.B_3\right|_{Q\times\left\{h\right\}}$ satisfying \eqref{eq:B3_via_M1M2M3} on $Q$ with some functions $M_1$, $M_2$, $M_3\in L^2\left(\mathbb{R}^2\right)$, there exists a unique extrapolant $\Bext$ satisfying \eqref{eq:B3_via_M1M2M3} on $\RR^2$ with the same $M_1$, $M_2$ and $M_3$.
\end{prop}
\begin{proof}
The existence of the extrapolant $\Bext$ is immediate from \eqref{eq:B3_via_M1M2M3} due do the fact that each of functions $\partial_{x_1}p_{h}$, $\partial_{x_2}p_{h}$ and $\partial_{h}p_{h}$ are defined on the whole $\RR^2$. Let us show the uniqueness. To this effect, it is useful to rewrite \eqref{eq:B3_via_M1M2M3} in an equivalent form
\begin{equation}
B_{3}\left(\bx,h\right)=-\frac{1}{2}\iint_{\RR^2}\partial_{h}p_{h}\left(\bx-\bt\right)f_{M}\left(\bt\right)\dd^{2}\bt=-\frac{1}{2}\left(\partial_{h}p_{h}\star f_M\right)\left(\bx\right),\qquad\bx\in\RR^2,\label{eq:B3_via_fM}
\end{equation}
with $f_M:=\mathcal{R}_1 M_1+\mathcal{R}_2 M_2+M_3\in L^2\left(\RR^2\right)$, where $\mathcal{R}_1$, $\mathcal{R}_2$ are Riesz transforms, see \cite[Thm 2.1]{mag_kern}, \cite[Sec 3.1]{mag_kern_mom}.
Suppose the contrary, i.e., there is some $L^2\left(\RR^2\right)\ni\widetilde{f}_M\neq f_M$ such that the corresponding field
\[
\widetilde{B}_{3}\left(\bx,h\right):=-\frac{1}{2}\left(\partial_{h}p_{h}\star \widetilde{f_M}\right)\left(\bx\right),\qquad\bx\in\RR^2,
\]
satisfies $\widetilde{B}_3\left(\cdot,h\right)=B_3\left(\cdot,h\right)=\Bm$ on $Q$ whereas $\widetilde{B_3^{\rm{ext}}}:=\widetilde{B}_3\left(\cdot,h\right)\neq B_3\left(\cdot,h\right)=:\Bext$ on $\RR^2\backslash{\overline{Q}}$. It is then immediate that the function $B_3^{\rm{diff}}:=\widetilde{B}_3\left(\cdot,h\right)-B_3\left(\cdot,h\right)=-\frac{1}{2}\partial_{h}p_{h}\star\left(\widetilde{f_M}-f_M\right)$ vanishes on $Q$. On the other hand, since  $B_3^{\rm{diff}}$ is real-analytic on $\RR^2$ (as the trace of a harmonic function), this function cannot vanish on an open set of $\RR^2$ unless it is identically zero on the whole $\RR^2$. But $-\frac{1}{2}\partial_{h}p_{h}\star$ is a positive operator on $L^2\left(\RR^2\right)$ (as it corresponds to the Fourier multiplier $\pi\left|\bf{k}\right|\exp\left(-2\pi h\left|\bf{k}\right|\right)$, $\bf{k}\in\RR^2$ is the Fourier domain variable).  
Therefore, we conclude that $\widetilde{f_M}-f_M\equiv 0$ which entails that $\widetilde{B}_3\left(\cdot,h\right)=B_3\left(\cdot,h\right)$ on $\RR^2$ thus contradicting the initial assumption $\widetilde{B_3^{\rm{ext}}}\neq\Bext$.
\end{proof}
Despite the positive affirmation of Proposition \ref{prop:ext_uniq}, any set of realistic measurement data is necessarily contaminated by presence of the noise. Hence, in devising an extrapolation method, instability with respect to small perturbations of the data should be also kept in mind. When using spectral approaches, stabilisation can be achieved by means of low-pass filtering as typical noise has a much larger high-frequency content compared to the exact field data. We shall employ this idea here by putting a spectral approach at the heart of the present extrapolation method.

First, we are going to describe the method from the viewpoint of practical implementation, and then outline its justification. 
For convenience, let us set:
\begin{equation*}
K_{12}:=-\frac{1}{2} p_h,\qquad  K_{3}:=-\frac{1}{2} \partial_h p_h,\qquad \left(K\star_Q f\right)\left(\bx\right)\equiv\iint_Q K\left(\bx-\bt\right) f\left(\bt\right)\dd^2\bt,
\end{equation*}
\begin{equation*}
\left\langle f,\, g\right\rangle\equiv\iint_Q f\left(\bt\right)g\left(\bt\right)\dd^2\bt,\qquad \left\Vert f\right\Vert\equiv\left\Vert f\right\Vert_{L^2\left(Q\right)}.
\end{equation*}
\subsection*{Algorithm}
\begin{enumerate}[label=(\roman*)]
\item Fix positive numbers $J$, $N\in\mathbb{N}_{+}$ sufficiently large (for ideal data $\Bm$).
\item Find $J$ 
 largest in modulus eigenvalues 
$\left|\mu_{1}\right|\geq\left|\mu_{2}\right|\geq\ldots\geq\left|\mu_{J}\right|>0$ and corresponding real-valued eigenfunctions $\left(\phi_j\right)_{j=1}^{J}\subset L^2\left(Q\right)$ satisfying $K_{12}\star_Q \phi_j=\mu_j \phi_j$ on $Q$ and $\left\Vert\phi_j\right\Vert=1$ for $j=1,\, \ldots,\, J$.
\item For any $f\in L^2\left(Q\right)$, define the operator $\mathcal{S}_{J}:\,L^2\left(Q\right)\rightarrow L^2\left(Q\right)$ as
\begin{equation}\label{eq:S_def}
\left(\mathcal{S}_{J}f\right)\left(\bx\right):=\iint_{Q}S_{J}\left(\bx,\bt\right)f\left(\bt\right)\dd^{2}\bt,\qquad S_{J}\left(\bx,\bt\right):=-\sum_{j=1}^{J}\frac{\left\langle \partial_{h}p_{h}\left(\cdot-\bx\right),\phi_{j}\right\rangle }{2\mu_{j}}\phi_j\left(\bt\right).
\end{equation}
\item Find $N$ 
largest eigenvalues 
$\lambda_{1}\geq\lambda_{2}\geq\ldots\geq\lambda_{N}$ and corresponding real-valued eigenvectors $\left(\left[\varphi_{12}^n,\,\varphi_{3}^n\right]^{T}\right)_{n=1}^{N}\subset \left[L^2\left(Q\right)\right]^2$ satisfying 
\begin{equation}
\begin{cases}
K_{12}\star_{Q}\varphi_{12}^{n}+K_{3}\star_{Q}\varphi_{3}^{n}=\lambda_{n}\varphi_{12}^{n}\\
K_{12}\star_{Q}\varphi_{12}^{n}+\mathcal{S}_{J}K_{3}\star_{Q}\varphi_{3}^{n}=\lambda_{n}\varphi_{3}^{n}
\end{cases}\text{ on }\,\,Q
\end{equation}
and  $\left\Vert\varphi_{12}^n\right\Vert^2+\left\Vert\varphi_{3}^n\right\Vert^2=1$ for $n=1,\, \ldots,\, N$.
\item Given $\Bm\in L^2\left(Q\right)$, construct the field extrapolant as
\begin{equation}
\Bext\left(\bx\right)=\sum_{n=1}^{N}b_{n}\widetilde{\varphi}_{12}^{n}\left(\bx\right),\qquad\bx\in\RR^2,
\end{equation}
where, for $n=1,\, \ldots,\, N$,
\begin{equation}
b_n:=\left\langle B_{3}^{\text{meas}},\varphi_{12}^{n}\right\rangle +\left\langle \mathcal{S}_{J}B_{3}^{\text{meas}},\varphi_{3}^{n}\right\rangle,
\end{equation}
\begin{equation}
\widetilde{\varphi}_{12}^{n}\left(\bx\right):=\frac{1}{\lambda_{n}}\iint_{Q}\left[K_{12}\left(\bx-\bt\right)\varphi_{12}^{n}\left(\bt\right)+K_{3}\left(\bx-\bt\right)\varphi_{3}^{n}\left(\bt\right)\right]\dd^{2}\bt,\qquad\bx\in\mathbb{R}^{2}.
\end{equation}
\end{enumerate}
\subsection*{Towards justification of the method}
Let us start by rewriting \eqref{eq:B3_via_M1M2M3} in a slightly different form employing integration by parts in the first 2 terms:
\begin{equation}
B_{3}\left(\bx,h\right)=\left(K_{12}\star_Q D_{M_{12}}\right)\left(\bx\right)+\left(K_{3}\star_Q M_{3}\right)\left(\bx\right),\qquad\bx\in\RR^2,
\end{equation}
where $D_{M_{12}}\left(\bt\right):=\partial_{t_{1}}M_{1}\left(\bt\right)+\partial_{t_{2}}M_{2}\left(\bt\right)\in L^2\left(Q\right)$. In particular, on $Q$, we have $\Bm\left(\bx\right)=\left(K_{12}\star_Q D_{M_{12}}\right)\left(\bx\right)+\left(K_{3}\star_Q M_{3}\right)\left(\bx\right)$ and applying to this equation the operator $\mathcal{S}_J$ defined in \eqref{eq:S_def}, we can reformulate both equations as
\begin{equation}\label{eq:vec_reform_init}
\left[\begin{array}{cc}
K_{12}\star_{Q} & K_{3}\star_{Q}\\
\mathcal{S}_J K_{12}\star_{Q}\hspace{1em} & \mathcal{S}_J K_{3}\star_{Q}
\end{array}\right]\left[\begin{array}{c}
D_{M_{12}}\\
M_{3}
\end{array}\right]=\left[\begin{array}{c}
\Bm\\
\mathcal{S}_J\Bm
\end{array}\right]\qquad \text{ on }\,\, Q.
\end{equation}
Since the truncated Poisson operator $p_h\star_Q$ is positive and compact on $L^2\left(Q\right)$, its range is dense in $L^2\left(Q\right)$. We hence have
\[
\iint_{Q}K_{12}\left(\bt-\by\right)S_{J}\left(\bx,\by\right)\dd^{2}\by=K_{3}\left(\bt-\bx\right)+r_{J}\left(\bt,\bx\right),\qquad \bx,\,\,\bt\in Q,
\]
\[
r_J\left(\bt,\bx\right):=\sum_{j=1}^J\left\langle K_3\left(\cdot-\bx\right),\phi_j\right\rangle \phi_j\left(\bt\right)-K_3\left(\bt-\bx\right),\qquad \bx,\,\,\bt\in Q,
\] 
with $\left\Vert r_J\left(\cdot,\bx\right)\right\Vert\rightarrow 0$ as $J\rightarrow \infty$ uniformly for any $\bx\in Q$.
Moreover, from the definition of $S_J$ in \eqref{eq:S_def}, it is easy to verify that
\[
\iint_{Q}S_{J}\left(\bx,\by\right)K_{3}\left(\by-\bt\right)\dd^{2}\by=\iint_{Q}K_{3}\left(\bx-\by\right)S_{J}\left(\bt,\by\right)\dd^{2}\by,\qquad \bx,\,\,\bt\in Q.
\]
Consequently, the matrix-operator $\mathcal{L}_{J}:\,\left[L^2\left(Q\right)\right]^2\rightarrow \left[L^2\left(Q\right)\right]^2$ defined as
\[
\mathcal{L}_{J}:=\left[\begin{array}{cc}
K_{12}\star_{Q} & K_{3}\star_{Q}\\
K_{3}\star_{Q} & \mathcal{S}_{J}K_{3}\star_{Q}
\end{array}\right]
\]
is self-adjoint, and \eqref{eq:vec_reform_init} rewrites as
\begin{equation}\label{eq:vec_reform}
\mathcal{L}_{J}\left[\begin{array}{c}
D_{M_{12}}\\
M_{3}
\end{array}\right]=\left[\begin{array}{c}
\Bm\\
\mathcal{S}_J\Bm
\end{array}\right]-\left[\begin{array}{c}
0\\
\iint_Q r_J\left(\bt,\cdot\right) D_{M_{12}}\left(\bt\right)\dd^2 \bt
\end{array}\right]\qquad \text{ on }\,\, Q,
\end{equation}
where the second term on the right-hand side can be made arbitrary small for sufficiently large $J$.  Since the operator $\mathcal{L}_{J}$ is also compact, the spectral theorem applies and motivates the algorithm above. Further details will be published elsewhere.

\section{Numerical illustration}
To test the method numerically, we consider the field $B_3^\textrm{meas}$ on $Q=\left[-1,1\right]^2$ which is produced by a magnetisation whose planar support consists of 4 rectangles contained in $Q$, see Figure \ref{fig:mag_B3_orig}. Choosing $N=J=80$ in the described above algorithm, we obtain an extrapolant $B_3^\textrm{ext}$. In Figure \ref{fig:ext_results}, $B_3^\textrm{ext}$ is shown on a rectangle $\left[-10,10\right]^2$, together with a pointwise error. The relative $L^2$--error in this region is around $7\%$. 
\begin{figure}
\includegraphics[scale=0.35]{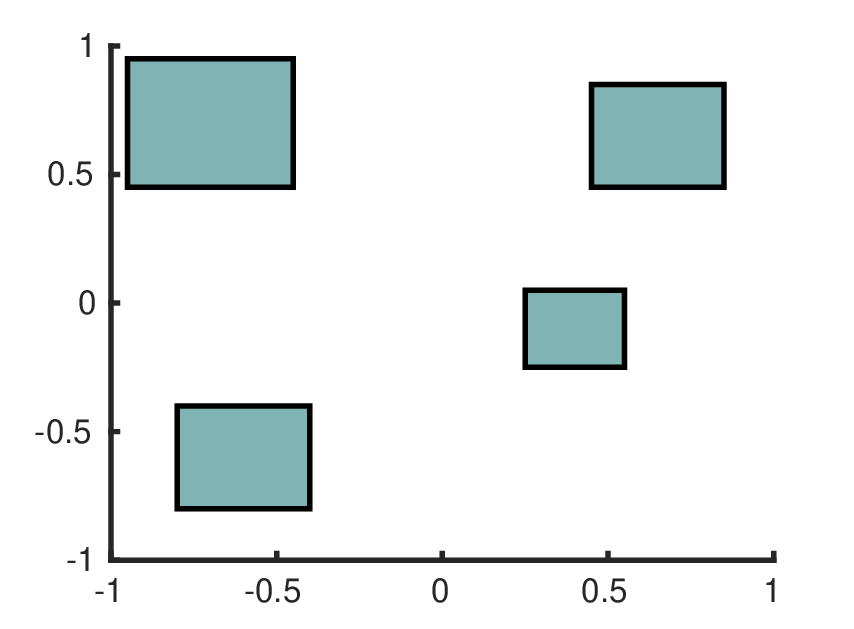}\qquad\includegraphics[scale=0.35]{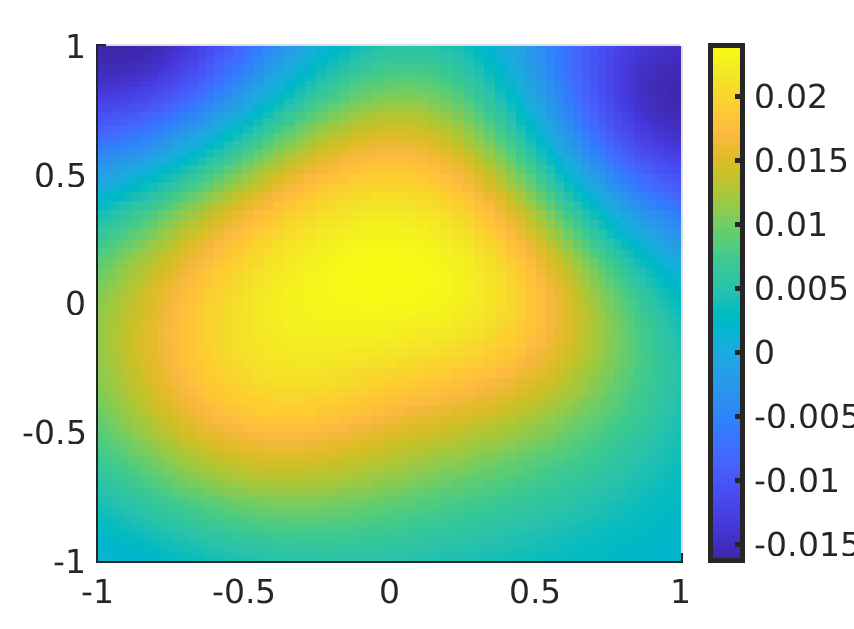}\caption{Planar magnetisation support (left) and measured data $B_3^\textrm{meas}$
(right) on $\left[-1,1\right]^2$}\label{fig:mag_B3_orig}
\end{figure}
\begin{figure}
\includegraphics[scale=0.35]{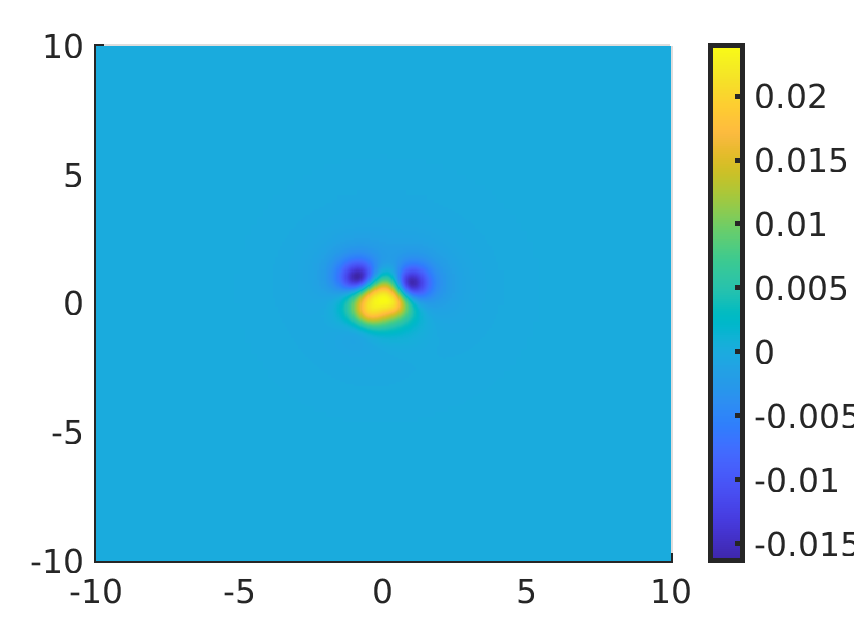}\qquad\includegraphics[scale=0.35]{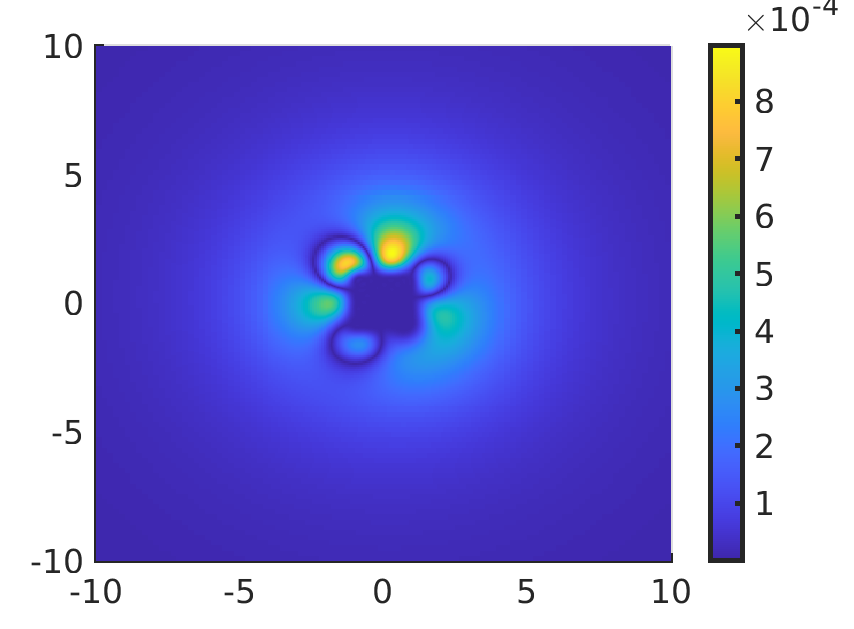}\caption{Extrapolated field $B_3^\textrm{ext}$ (left) and extrapolation error
(right) on $\left[-10,10\right]^2$}\label{fig:ext_results}
\end{figure}
\section{Conclusions}
Motivated by improvement of various aspects of inverse magnetisation problem, we have considered an issue of the field extrapolation. After proving the uniqueness of the extrapolant (for the ideal data), we have proposed an approach which reduces the problem by means of solving an auxiliary equation. The method is constructed based on spectral decompositions of the operators of the reduced and the auxiliary problems and hence can be referred to as a "double-spectral" approach. Further analysis and details on the choice of spectral truncation parameters $J$ and $N$ featuring in the presented algorithm will be discussed in a separate paper.
\section*{Acknowledgement}
The author is grateful for motivating discussions with Juliette Leblond (Centre Inria d'Universit{\'e} C{\^o}te d'Azur, France) and Eduardo Andrade Lima (MIT, USA).

\end{document}